\documentclass{jhrs}

\usepackage{amsmath}
\usepackage[all]{xy}
\usepackage{pstricks,pst-node,pst-plot}

\newtheorem{thm}{Theorem}[section]

\theoremstyle{definition}

\newtheorem*{thankyou}{Acknowledgements}
\newtheorem{lem}[thm]{Lemma}
\newtheorem{defn}[thm]{Definition}
\newtheorem{rem}[thm]{Remark}
\theoremstyle{remark}

\numberwithin{equation}{section}

\begin{document}
\title{Popaths and Holinks}          

\author{David A. Miller}             
\email{d.miller@maths.abdn.ac.uk}       
\address{Mathematics Department\\ 
         University of Aberdeen\\
         Aberdeen\\
         United Kingdom\\
         AB24 5UE}



\begin{abstract}
In the study of stratified spaces it is useful to examine spaces of
popaths (paths which travel from lower strata to higher strata) and
 holinks (those spaces of popaths which immediately leave a lower stratum
for their final stratum destination). It is not immediately clear
that for adjacent strata these two path spaces are homotopically
equivalent, and even less clear that this equivalence can be
constructed in a useful way (with a deformation of the space of
popaths which fixes start and end points and where popaths instantly
become members of the holink). The advantage of such an equivalence
is that it allows a stratified space to be viewed categorically
because popaths, unlike holink paths (which are easier to study),
can be composed. This paper proves the aforementioned equivalence in
the case of Quinn's homotopically stratified spaces \cite{FQ}.
\end{abstract}

\received{Month Day, Year}   
\revised{Month Day, Year}    
\published{Month Day, Year}  
\submitted{Name of Editor}  

\volumeyear{2006} 
\volumenumber{1}  
\issuenumber{1}   

\startpage{1}     

\maketitle

\section{Introduction}

\noindent There are many different notions of a stratified space.
One such notion is that of a homotopically stratified space. These
were introduced by Frank Quinn. Here the strata are related by
``homotopy rather than geometric conditions" \cite{FQ}. This makes
them ideal for studying the topology of stratified spaces. Two such
tools for studying that topology are holinks and popath spaces.

\smallskip

\noindent The popaths between two strata are any paths which travel
from one stratum to the other passing only from ``lower strata to
higher strata". On the other hand the holink between two strata
consists only of popaths which instantly leave one stratum for the
other.

\smallskip

\noindent Popaths are very useful in obtaining a categorical view of
stratified spaces. However for a space with many strata the space of
popaths between two strata could be difficult to compute or
visualize. The holink between two strata only depends on the two
strata involved and with this in mind is easier to deal with, but
problems may arises because holink paths can't be composed.
Therefore a result connecting these concepts becomes desirable.

\smallskip

\noindent The result obtained here is that for the space of popaths
and the holink between two fixed strata there exists a homotopy
$h:\mathrm{popaths}\times I\rightarrow \mathrm{popaths}$ which fixes
the start and end points of paths, $\mathrm{image}\{h_{s}\}\subset
\mathrm{holink}$ when $s\in(0,1]$ and $h_{0}=\mathrm{identity}$. It
may seem strange to require a result which is stronger than just the
inclusion being a homotopy equivalence, but exactly this result has
already found a relevance in the work of Jon Woolf \cite{JW} and
seems to be required to construct a particular map to prove that a
stratified space can be reconstructed from its popath category (the
author hopes to show this in the near future).

\smallskip

\thankyou{I would like to take this opportunity to acknowledge the
continuing guidance and assistance of my PhD supervisor Michael
Weiss for which I am very grateful. I would also like to thank Jon
Woolf for many useful comments and correspondences.}

\

\section{Homotopically Stratified Spaces}

\begin{defn} A topological space $X$ is \textbf{filtered} if there are
designated closed subspaces $X^{i}$ indexed by a finite poset
$S_{X}$ such that $X=\bigcup_{i\in S_{X}}X^{i}$ and $X^{j}\subseteq
X^{i}\Leftrightarrow j\leq i$. The \textbf{strata} are defined as
path connected components of $X_{i}=X^{i}-\bigcup_{j<i}X^{j}$. If
$j<i$ we may say $X_{i}$ is a higher stratum than $X_{j}$ or
equivalently $X_{j}$ is a lower stratum than $X_{i}$.
\end{defn}

\smallskip

\begin{defn} A closed subspace $K$ of a filtered space is said to be \textbf{pure} if
it is a union of strata.
\end{defn}

\smallskip

\begin{defn}Assume $W$ is a subspace of a filtered space and $W$ contains the distinct strata $X_{b}$ and $X_{a}$. Let the space of \textbf{popaths}
from $X_{a}$ through $W$ to $X_{b}$ be denoted by
$\mathrm{pop}(X_{b},W,X_{a})$ and defined as the space (with the
compact open topology) of all order preserving paths
$\omega:[0,T_{\omega}]\rightarrow W$ such that $\omega(0)\in X_{a}$,
$\omega(T_{\omega})\in X_{b}$. Here order preserving means if
$t_{1}\leq t_{2}$ and $\omega(t_{1})\in X_{j}$ while
$\omega(t_{2})\in X_{i}$ then $j\leq i$ (meaning $\omega$ cannot
flow from higher strata into lower strata).
\end{defn}

\smallskip

\begin{defn} For a space $S$ and subspace $Y\subset S$ the \textbf{holink} (also called \textbf{homotopy link})
between $Y$ and $S-Y$
 denoted $\mathrm{hol}(S,Y)$ is defined as the space (with the compact open topology) of paths
$\omega:[0,T_{\omega}]\rightarrow S$ where $\omega(0)\in Y$ and
$\omega(t)\in S-Y$ for $t\in(0,T_{\omega}]$.
\end{defn}

\smallskip

\begin{rem} It should be clear that $\mathrm{hol}(X_{b}\cup X_{a},X_{a})$ is the
subspace of $\mathrm{pop}(X_{b},X,X_{a})$ consisting of popaths
which immediately leave $X_{a}$ and travel straight into $X_{b}$.
\end{rem}

\smallskip

\begin{defn}Suppose $K$, $L$ are unions of strata in $X$ and
$L\subseteq K$. Then $L$ is said to be \textbf{tame} in $K$ if there
is a neighborhood $N$ of $L$ in $K$ and a nearly stratum preserving
 strong deformation retraction $r$ of $N$ onto $L$. Here \textbf{nearly
stratum preserving} means points of $K-L$ remain in the same stratum
until the last moment when they get pushed into $L$.
\end{defn}

\smallskip

\begin{defn}Let $R:N-L\rightarrow \mathrm{hol}(N,L)$ be defined by $R(x)(t)=r(x,1-t)$ for all $x\in N-L$.
\end{defn}

\smallskip

\begin{defn}A filtered metric space $X$ is a \textbf{homotopically stratified space}
if for every $j<i$, $X_{j}$ is tame in $X_{j}\cup X_{i}$ and the map
from $\mathrm{hol}(X_{i}\cup X_{j},X_{j})$ to $X_{j}$ given by
evaluation at the start point is a fibration.
\end{defn}

\

\section{Popaths and Holinks in the 2 Strata Case}

\begin{defn} Let $\kappa:\mathrm{pop}(X_{b},X_{b}\cup
X_{a},X_{a})\rightarrow \mathbb{R}$ be the map which sends
$\omega\in \mathrm{pop}(X_{b},X_{b}\cup X_{a},X_{a})$ to the unique
point, $\kappa(\omega)$, in $\mathbb{R}$ such that $\omega(t)\in
X_{a}$ for $t\leq\kappa(\omega)$ and $\omega(t)\in X_{b}$ for
$t>\kappa(\omega)$. Note $\kappa$ is not a continuous map, but it is
upper-continuous. We define upper-continuity as meaning for any
$\omega\in \mathrm{pop}(X_{b},X_{b}\cup X_{a},X_{a})$ and any
neighborhood $V$ of $\kappa(\omega)$ having the form $[0,r)$ there
exists a neighborhood U of $\omega$ such that $f(U)\subset V$.
\end{defn}

\smallskip

\begin{lem}Let $N$ be a neighborhood of tameness for $X_{a}$ in $X_{a}\cup X_{b}$ where $X_{a}, X_{b}$ are distinct.
There exists a continuous map $\lambda:\mathrm{pop}(X_{b},X_{b}\cup
X_{a},X_{a})\rightarrow \mathbb{R}$ that for any popath
$\omega:[0,T_{\omega}]\rightarrow X$ satisfies:
\begin{enumerate}
\item $\lambda(\omega)\in(\kappa(\omega),T_{\omega})$
\item $\omega(t)\in N-X_{a}$
for $t\in(\kappa(\omega),\lambda(\omega)]$
\end{enumerate}
\end{lem}

\smallskip

\begin{rem}This will be useful
when we wish to use the $\kappa$ map but cannot because it is not
continuous.
\end{rem}

\begin{proof}Since $X$ is metric, $\mathrm{pop}(X_{b},X_{b}\cup
X_{a},X_{a})$ is also metric and so paracompact. Therefore by a
partition of unity type argument it suffices to show it is true
locally.

\smallskip

\noindent Fix $\sigma\in \mathrm{pop}(X_{b},X_{b}\cup X_{a},X_{a})$,
clearly we can choose a point $\Gamma$ in $[0,T_{\sigma}]$ which
satisfies the conditions of $\lambda(\sigma)$. Now since $\kappa$ is
upper continuous the same value $\Gamma$ satisfies the conditions to
be $\lambda(\omega)$ for all $\omega$ within a small enough
neighborhood of $\sigma$. Hence the lemma holds locally and
therefore holds globally.
\end{proof}

\smallskip

\begin{defn}Let E denote the space
\begin{displaymath}
\{(\omega,t)\in \mathrm{pop}(X_{b},X_{b}\cup
X_{a},X_{a})\times\mathbb{R}\,:\,t\in(\kappa(\omega),\lambda(\omega)]\}
\end{displaymath}
and $p$ denote the canonical map $p:E\rightarrow
\mathrm{pop}(X_{b},X_{b}\cup X_{a},X_{a})$. In fact $p$ is a fiber
bundle homeomorphic to the trivial bundle with total space
$\mathrm{pop}(X_{b},X_{b}\cup X_{a},X_{a})\times (0,1]$. A
trivialisation is given by
$(\omega,t)\mapsto(\omega,\beta_{\omega}(t))$, where
$\beta_{\omega}(t)$ is the unique member of $(0,1]$ such that
\begin{displaymath}
\int^{t}_{\kappa(\omega)}\mathrm{dist}_{X_{a}}(\omega(t^{\prime}))dt^{\prime}=\beta_{\omega}(t)\int^{\lambda(\omega)}_{\kappa(\omega)}\mathrm{dist}_{X_{a}}(\omega(t^{\prime}))dt^{\prime}.
\end{displaymath}
\end{defn}

\smallskip

\begin{lem} The inclusion map $i$ induces a homotopy equivalence
\begin{displaymath}
\mathrm{hol}(X_{b}\cup X_{a},X_{a})\simeq
\mathrm{pop}(X_{b},X_{b}\cup X_{a},X_{a}).
\end{displaymath}
Furthermore it has a homotopy inverse
$\varphi:\mathrm{pop}(X_{b},X_{b}\cup X_{a},X_{a})\rightarrow
\mathrm{hol}(X_{b}\cup X_{a},X_{a})$ where there exists a homotopy
$h$ from the identity map to $i\circ \varphi$ which fixes the start
and end points of paths and $h_{s}(\omega)\in \mathrm{hol}(X_{b}\cup
X_{a},X_{a})$ for all popaths $\omega$ when $s\in(0,1]$.
\end{lem}

\smallskip

\begin{rem} This lemma means in the two strata case we have a continuous way of
deforming the space of popaths into the holink so that popaths
instantly become holink paths.
\end{rem}

\begin{proof}We will in fact directly construct the maps $h_{s}:\mathrm{pop}(X_{b},X_{b}\cup X_{a},X_{a})
\rightarrow \mathrm{pop}(X_{b},X_{b}\cup X_{a},X_{a})$ which fix
start and end points, have image in $\mathrm{hol}(X_{b}\cup
X_{a},X_{a})$ when $s\in(0,1]$ and where $h_{0}$ is the identity
map. Then by setting $\varphi=h_{1}$ we have proved the lemma.

\smallskip

\noindent Let $N$ be a neighborhood of tameness for $X_{a}$ in
$X_{a}\cup X_{b}$ and $r$ be the corresponding strong deformation
retraction. Define a map $F:\mathrm{pop}(X_{b},X_{b}\cup
X_{a},X_{a})\times (0,1]\times I\rightarrow X_{a}$ by for all
$\omega\in \mathrm{pop}(X_{b},X_{b}\cup X_{a},X_{a})$ sending
$(\omega,s,t)$ to
$r\left(\omega(t\cdot\beta_{\omega}^{-1}(s)),1\right)$. Define
another map $G:\mathrm{pop}(X_{b},X_{b}\cup
X_{a},X_{a})\times(0,1]\rightarrow \mathrm{hol}(X_{b}\cup
X_{a},X_{a})$ by $\omega\mapsto
R\left(\omega(\beta_{\omega}^{-1}(s))\right)$ (see Definition 1.7).
The definition of Quinn stratified spaces tells us the map $E_{0}$,
evaluating a holink path at its start point, is a fibration. So
there is a lift $\widetilde{F}$,

\smallskip

$$\xymatrix{\mathrm{pop}(X_{b},X_{b}\cup
X_{a},X_{a})\times (0,1]\times\{1\}\ar[r]^-{G}\ar[d]&\mathrm{hol}(X_{b}\cup
X_{a},X_{a})\ar[d]^-{E_{0}}\\\mathrm{pop}(X_{b},X_{b}\cup X_{a},X_{a})\times (0,1]\times
I\ar[r]^-{F}\ar@{-->}[ru]^-{\widetilde{F}}&X_{a}}$$

\
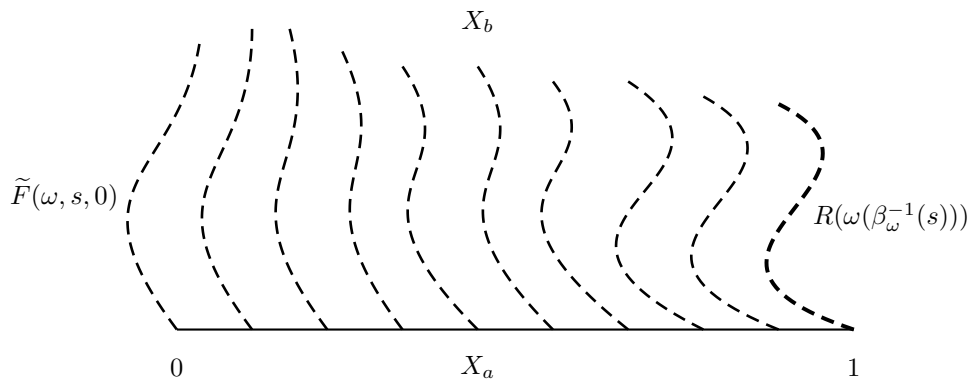
\begin{center}
\begin{pspicture}(0,0)(10,5)

\psbezier[linestyle=dashed,linewidth=1.5pt](10,0)(7,1)(11,2)(9,3)
\qline(1,0)(10,0) \rput[c](1,-.5){$0$} \rput[c](10,-.5){$1$}
\psbezier[linestyle=dashed,linewidth=1pt](9,0)(6,1.2)(10,2)(8,3.1)
\psbezier[linestyle=dashed,linewidth=1pt](8,0)(5,1.5)(9,2)(7,3.3)
\psbezier[linestyle=dashed,linewidth=1pt](7,0)(4.5,2)(7,2)(6,3.3)
\psbezier[linestyle=dashed,linewidth=1pt](6,0)(4,2)(6,2)(5,3.5)
\psbezier[linestyle=dashed,linewidth=1pt](5,0)(3,2)(5,2)(4,3.5)
\psbezier[linestyle=dashed,linewidth=1pt](4,0)(2.5,2)(4,2)(3.2,3.7)
\psbezier[linestyle=dashed,linewidth=1pt](3,0)(1.5,2)(3,2)(2.5,4)
\psbezier[linestyle=dashed,linewidth=1pt](2,0)(0.5,2)(2,2)(2,4)
\psbezier[linestyle=dashed,linewidth=1pt](1,0)(-0.5,2)(1,2)(1.3,3.8)
\rput[c](10.5,1.5){$R(\omega(\beta_{\omega}^{-1}(s)))$}
\rput[c](-0.5,1.8){$\widetilde{F}(\omega,s,0)$}
\rput[c](5,-.5){$X_{a}$} \rput[c](5,4.1){$X_{b}$}

\end{pspicture}
\end{center}

\

\

\noindent Now we can use $\widetilde{F}$ to construct to $h$. When
$s=0$, $h_{s}$ is the identity and when $s\in(0,1]$

$$h_{s}(\omega)(t)=\left\{\begin{array}{ll}\widetilde{F}\left(\omega,s,\frac{t+s\beta_{\omega}^{-1}(s)}{\beta_{\omega}^{-1}(s)}(1-s)\right)(\frac{t}{\beta_{\omega}^{-1}(s)}s)&0\leq
t\leq(1-s)\beta_{\omega}^{-1}(s)
\\\widetilde{F}\left(\omega,s,\frac{t}{\beta_{\omega}^{-1}(s)}\right)\left(\frac{s(s-1)(t-\beta_{\omega}^{-1}(s))+(t-(1-s)\beta_{\omega}^{-1}(s))}{s\beta_{\omega}^{-1}(s)}\right)&(1-s)\beta_{\omega}^{-1}(s)< t\leq\beta_{\omega}^{-1}(s)
\\\omega(t)&\beta_{\omega}^{-1}(s)\leq t\leq
T_{\omega}\end{array}\right..
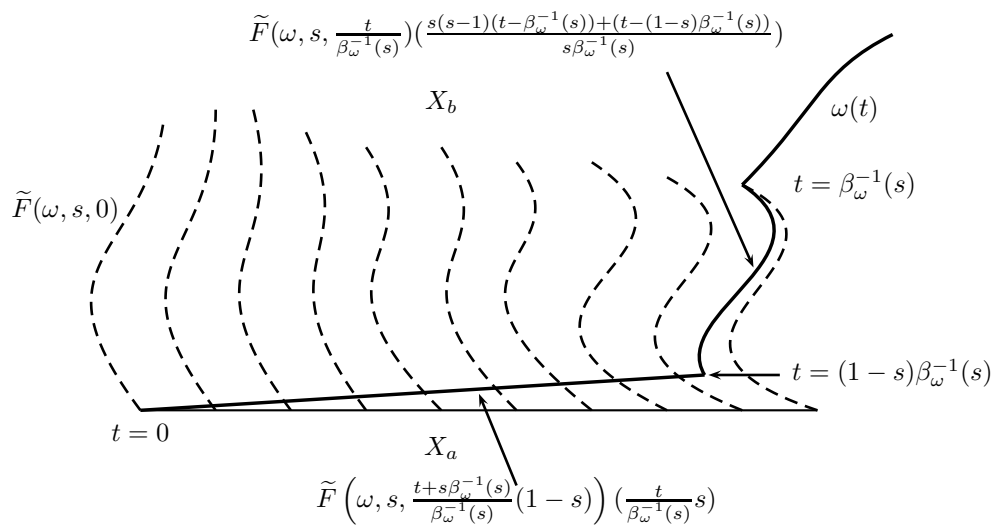
$$
\begin{center}
\begin{pspicture}(1.1,-.9)(10,6.1)

\psbezier[linestyle=dashed,linewidth=1pt](10,0)(7,1)(11,2)(9,3)
\qline(1,0)(10,0) \rput[c](1,-.3){$t=0$}
\rput[c](11,0.5){$t=(1-s)\beta_{\omega}^{-1}(s)$}

\psbezier[linestyle=dashed,linewidth=1pt](9,0)(6,1.2)(10,2)(8,3.1)
\psbezier[linestyle=dashed,linewidth=1pt](8,0)(5,1.5)(9,2)(7,3.3)
\psbezier[linestyle=dashed,linewidth=1pt](7,0)(4.5,2)(7,2)(6,3.3)
\psbezier[linestyle=dashed,linewidth=1pt](6,0)(4,2)(6,2)(5,3.5)
\psbezier[linestyle=dashed,linewidth=1pt](5,0)(3,2)(5,2)(4,3.5)
\psbezier[linestyle=dashed,linewidth=1pt](4,0)(2.5,2)(4,2)(3.2,3.7)
\psbezier[linestyle=dashed,linewidth=1pt](3,0)(1.5,2)(3,2)(2.5,4)
\psbezier[linestyle=dashed,linewidth=1pt](2,0)(0.5,2)(2,2)(2,4)

\psbezier[linestyle=dashed,linewidth=1pt](1,0)(-0.5,2)(1,2)(1.3,3.8)
\rput[c](10.5,3){$t=\beta_{\omega}^{-1}(s)$}
\rput[c](0,2.7){$\widetilde{F}(\omega,s,0)$}
\rput[c](5,-.5){$X_{a}$} \rput[c](5,4.1){$X_{b}$}

\psbezier[linewidth=1.3pt](1,0)(1,0)(1,0)(8.5,0.47)
\psbezier[linewidth=1.3pt](8.5,0.47)(8,1.3)(10.35,2.1)(9,3)
\psbezier[linewidth=1.3pt](9,3)(10,4)(10,4.5)(11,5)

\psline[linewidth=1pt]{->}(6,-1)(5.5,0.22)
\psline[linewidth=1pt]{->}(9.5,0.47)(8.55,0.47)
\psline[linewidth=1pt]{->}(8,4.5)(9.15,1.9)

\rput[c](6,-1.2){$\widetilde{F}\left
(\omega,s,\frac{t+s\beta_{\omega}^{-1}(s)}{\beta_{\omega}^{-1}(s)}(1-s)\right
)(\frac{t}{\beta_{\omega}^{-1}(s)}s)$}
\rput[c](6,5){$\widetilde{F}(\omega,s,\frac{t}{\beta_{\omega}^{-1}(s)})(\frac{s(s-1)(t-\beta_{\omega}^{-1}(s))+(t-(1-s)\beta_{\omega}^{-1}(s))}{s\beta_{\omega}^{-1}(s)})$}
\rput[c](10.5,4){$\omega(t)$}
\end{pspicture}
\end{center}

\noindent Clearly this is continuous within the three intervals for
$s\in(0,1]$. To see it is continuous at
$t=(1-s)\beta_{\omega}^{-1}(s)$ just substitute for $t$ and get
$\widetilde{F}\left(\omega,s,1-s\right)(s(1-s))$ for both
expressions. To see it is continuous at $t=\beta_{\omega}^{-1}(s)$
substitute into the second expression to get
$\widetilde{F}\left(\omega,s,1\right)(1)$, now
\begin{displaymath}
\widetilde{F}\left(\omega,s,1\right)(1)=G\left(\omega,s\right)(1)=R\left(\omega(\beta_{\omega}^{-1}(s))\right)(1)=\omega(\beta_{\omega}^{-1}(s))
\end{displaymath}
so it is continuous at $t=\beta_{\omega}^{-1}(s)$.

\smallskip

\noindent To prove $h_{s}\rightarrow Id$ as $s\rightarrow 0$ we will
prove it is true in each of the three intervals in the definition of
$h_{s}$ (in the last interval there is nothing to prove).

\noindent In the first interval
\begin{align*}
&\lim_{s\rightarrow0}\widetilde{F}\left (\omega,s,\frac{t+s\beta_{\omega}^{-1}(s)}{\beta_{\omega}^{-1}(s)}(1-s)\right )(\frac{t}{\beta_{\omega}^{-1}(s)}s)\\
=&\lim_{s\rightarrow0}\widetilde{F}\left(\omega,s,\frac{t}{\beta_{\omega}^{-1}(s)}\right)(0)\\
=&\lim_{s\rightarrow0}F\left(\omega,s,\frac{t}{\beta_{\omega}^{-1}(s)}\right)\\
=&\lim_{s\rightarrow0}r\left(\omega(\frac{t}{\beta_{\omega}^{-1}(s)}\beta_{\omega}^{-1}(s))\right),1)\\
=&\lim_{s\rightarrow0}r(\omega(t)),1),
\end{align*}
and since we are within $0\leq t\leq (1-s)\beta_{\omega}^{-1}(s)$
which tends towards $0\leq t\leq \kappa(\omega)$ as $s$ tends to $0$
then $r(\omega(t)),1)$ tends to $\omega(t)$ as $s$ tends to $0$ for
$0\leq t\leq (1-s)\beta_{\omega}^{-1}(s)$.

\smallskip

\noindent The second interval tends to
$t=\lim_{s\rightarrow0}\beta_{\omega}^{-1}(s)=\kappa(\omega)$ as $s$
tends to $0$ and
\begin{align*}
&\lim_{s\rightarrow0}\widetilde{F}\left(\omega,s,\frac{t}{\beta_{\omega}^{-1}(s)}\right)\left(\frac{s(s-1)(t-\beta_{\omega}^{-1}(s))+(t-(1-s)\beta_{\omega}^{-1}(s))}{s\beta_{\omega}^{-1}(s)}\right)\\
=&\lim_{s\rightarrow0}\widetilde{F}\left(\omega,s,1\right)(q)\\
=&\lim_{s\rightarrow0}G\left(\omega,s\right)(q)\\
=&\lim_{s\rightarrow0}R\left(\omega(\beta_{\omega}^{-1}(s))\right)(q),
\end{align*}
where
$q=\frac{s(s-1)(t-\beta_{\omega}^{-1}(s))+(t-(1-s)\beta_{\omega}^{-1}(s))}{s\beta_{\omega}^{-1}(s)}$.
Now $\beta_{\omega}^{-1}(s)\rightarrow \kappa(\omega)$ as
$s\rightarrow 0$  so
$r\left(\omega(\beta_{\omega}^{-1}(s)),-\right)$ tends to the
constant path $\omega(\kappa(\omega))$. Therefore
$\lim_{s\rightarrow0}R\left(\omega(\beta_{\omega}^{-1}(s))\right)(q)=\omega(\kappa(\omega))$.

\smallskip

\noindent This $h$ satisfies our requirements as detailed at the
start of the proof.\qedhere

\end{proof}

\section{Popaths and Holinks in the General Case}

\noindent In a slight abuse of notation we will throughout this
section redefine and use symbols like $\kappa, N, \lambda, E$ and
$\beta$ in a more general context.

\begin{defn} Consider $X_{a}, X_{b}$ to be distinct strata in a homotopically stratified space $X$. Let $\kappa:\mathrm{pop}(X_{b},X,X_{a})\rightarrow \mathbb{R}$ be the map which sends
$\omega\in \mathrm{pop}(X_{b},X,X_{a})$ to the unique point,
$\kappa(\omega)$, in $\mathbb{R}$ such that $\omega(t)\in X_{a}$ for
$t\leq\kappa(\omega)$ and $\omega(t)\not\in X_{a}$ for
$t>\kappa(\omega)$. Note $\kappa$ is not a continuous map, but as in
Definition 2.1 it is upper-continuous.
\end{defn}

\begin{lem} Suppose $X$ is a homotopically stratified space and $K\subset
X$ is pure. Then there is a nearly stratum preserving strong
deformation retract $r$ of a neighborhood $N$ of $K$ in $X$ to $K$.
The neighborhood may be referred to as a neighborhood of tameness.
\end{lem}

\begin{proof}
This is the first part of Proposition 3.2 of ``Homotopically
Stratified Sets" by Frank Quinn \cite{FQ}.
\end{proof}

\smallskip

\begin{lem}Consider $X_{a}$ as a lowest possible stratum in $X$ by
if necessary discarding any lower strata (which are of no
consequence when considering $\mathrm{pop}(X_{b},X,X_{a})$). Let $N$
be a neighborhood of tameness for $X_{a}$ in $X$. There exists a
continuous map $\lambda:\mathrm{pop}(X_{b},X,X_{a})\rightarrow
\mathbb{R}$ that for all popaths $\omega:[0,T_{\omega}]\rightarrow
X$ satisfies:
\begin{enumerate}
\item $\lambda(\omega)\in(\kappa(\omega),T_{\omega})$
\item $\omega(t)\in N-X_{a}$
for $t\in(\kappa(\omega),\lambda(\omega)]$
\end{enumerate}
\end{lem}

\begin{rem}Again this will be
useful when we wish to use the $\kappa$ map but cannot because it is
not continuous.
\end{rem}

\begin{proof}
This is proved in exactly the same way as Lemma 2.2.
\end{proof}

\smallskip

\begin{defn}Let E denote the space $\{(\omega,t)\in \mathrm{pop}(X_{b},X,X_{a})
\times\mathbb{R}\,:\,t\in(\kappa(\omega),\lambda(\omega)]\}$ and $p$
denote the canonical map $p:E\rightarrow
\mathrm{pop}(X_{b},X,X_{a})$. In fact $p$ is a fiber bundle
homeomorphic to the trivial bundle with total space
$\mathrm{pop}(X_{b},X,X_{a})\times (0,1]$. Let $\beta_{\omega}$
denote a reparametrization of $(\kappa(\omega),\lambda(\omega)]$ to
$(0,1]$ giving a trivialization. The trivialisation can be obtained
in the same way as in Definition 2.4.
\end{defn}

\smallskip

\begin{defn}Let $B$ be a path space and $A$ a subspace of $B$. We
will say the inclusion $A\subset B$ is a special path inclusion if
there exists a homotopy $h:B\times I\rightarrow B$ which fixes the
start and end points of paths, $\mathrm{image}\{h_{s}\}\subset A$
when $s\in(0,1]$ and $h_{0}=\mathrm{identity}$. Note this implies
the inclusion map is a homotopy equivalence. Also note the
composition (not concatenation) of two special path inclusions is
again a special path inclusion.
\end{defn}

\smallskip

\begin{rem}Lemma 2.5 proves that $\mathrm{hol}(X_{b}\cup X_{a},X_{a})\subset \mathrm{pop}(X_{b},X_{b}\cup
X_{a},X_{a})$ is a special path inclusion. The aim of this paper is
to show that $\mathrm{hol}(X_{b}\cup X_{a},X_{a})\subset
\mathrm{pop}(X_{b},X,X_{a})$ is a special path inclusion for any two
strata $X_{a}$ and $X_{b}$ of a homotopically stratified space $X$.
\end{rem}

\smallskip

\begin{lem}Consider $X_{b}$ as any
stratum in $X$ and $X_{a}$ as a lowest possible stratum in $X$ by if
necessary discarding any lower strata. If $\mathrm{hol}(X_{b}\cup
X_{a},X_{a})\subset \mathrm{pop}(X_{b},X,X_{a})$ is a special path
inclusion then
\begin{displaymath}
\{\sigma\in \mathrm{pop}(X_{b},X,X)\,:\,\sigma(\delta)\not\in
X_{a}\text{ for all } \delta>0\}\subset \mathrm{pop}(X_{b},X,X)
\end{displaymath}
is also a special path inclusion.
\end{lem}

\begin{proof}First let us show $\{\sigma\in
\mathrm{pop}(X_{b},X,N)\,:\,\sigma(\delta)\not\in X_{a}$ for all
$\delta>0\}\subset \mathrm{pop}(X_{b},X,N)$ is a special path
inclusion for $N$ a neighborhood of tameness of $X_{a}$ in $X$.

\smallskip

\noindent Given a path $\sigma\in \mathrm{pop}(X_{b},X,N)$ define
$\sigma^{+}\in \mathrm{pop}(X_{b},X,X_{a})$ by
$$\sigma^{+}(t)=\left\{\begin{array}{ll}r(\sigma(0),1-t)&0\leq t\leq 1\\\sigma(t-1)&1\leq t\leq
T_{\omega}+1\end{array}\right..$$

\smallskip

\noindent Let f be the map from $\mathrm{pop}(X_{b},X,X_{a})\times
I$ to Moore paths onto $I$ defined by

$$f(\omega,s)(t)=\left\{\begin{array}{ll}0&0\leq t\leq 1
\\\left(\min\{1,(t-1)(\lambda(\omega)-t)\}\right)s&1\leq t\leq\lambda(\omega)
\\0&\lambda(\omega)\leq t\leq T_{\omega}
\end{array}\right.$$

\noindent for all $\omega\in \mathrm{pop}(X_{b},X,X_{a})$ where
$1<\lambda(\omega)$ and defined by $f(\omega,s)(t)=0$ for all
$t\in[0,T_{\omega}]$ if $\lambda(\omega)\leq1$.

\

\noindent Now define a suitable homotopy for special path inclusion
$g:\mathrm{pop}(X_{b},X,N)\times I\rightarrow
\mathrm{pop}(X_{b},X,N)$ by
$g_{s}(\sigma)(t)=\left(h_{f(\sigma^{+},s)(t+1)}\sigma^{+}\right)(t+1)$
where $h$ is a homotopy for the special path inclusion
$\mathrm{hol}(X_{b}\cup X_{a},X_{a})\subset
\mathrm{pop}(X_{b},X,X_{a})$. Intuitively this can be thought of as
concatenating a path $[0,1]\rightarrow X$ to the start of $\sigma$
then manipulating the $(1,\lambda(\sigma))$ part of the new path
away from $X_{a}$ using $h$ and finally removing the $[0,1)$ part
that was added.

\smallskip

\noindent Now since for all $s$, $g_{s}(\sigma)(t)=\sigma(t)$ when
$t\geq t_{0}$ for any $t_{0}$ where $\sigma(t_{0})\not\in N$ we can
extend $g$ to give a special path inclusion $\{\sigma\in
\mathrm{pop}(X_{b},X,X)\,:\,\sigma(\delta)\not\in X_{a}$ for all
$\delta>0\}\subset \mathrm{pop}(X_{b},X,X)$ by setting
$g_{s}(\sigma)$ as $\sigma$ when $\sigma\in
\mathrm{pop}(X_{b},X,X-N)$.
\end{proof}

\smallskip

\begin{thm}$\mathrm{hol}(X_{b}\cup X_{a},X_{a})\subset
\mathrm{pop}(X_{b},X,X_{a})$ is a special path inclusion for any two
distinct strata $X_{a}$ and $X_{b}$ of a homotopically stratified
space $X$.
\end{thm}

\begin{proof}In the case when $X$ only has two strata then
$X=X_{a}\cup X_{b}$ and the proposition is proved as Lemma 2.5.
Likewise if there are no other strata which popaths from $X_{a}$ to
$X_{b}$ can pass through then
$\mathrm{pop}(X_{b},X,X_{a})=\mathrm{pop}(X_{b},X_{a}\cup
X_{b},X_{a})$ and the proposition is again proved by Lemma 2.5.
Therefore we can assume $X_{a}$ is a lowest stratum in $X$ and there
is at least one other stratum, $X_{c}$, which popaths from $X_{a}$
to $X_{b}$ may pass through, let $X_{c}$ denote a lowest stratum of
this type.

\smallskip

\noindent We will assume inductively that the lemma holds for
$\mathrm{hol}(X_{b}\cup X_{c},X_{c})\subset
\mathrm{pop}(X_{b},X,X_{c})$ and using this show
$\mathrm{pop}(X_{b},X-X_{c},X_{a})\subset
\mathrm{pop}(X_{b},X,X_{a})$ is a special path inclusion. Then the
proposition can be proven by removing strata like $X_{c}$ until we
are in the two strata case which has been proved already as Lemma
2.5.

\smallskip

\noindent We will define the homotopy we require as a composition
(not concatenation) of two homotopies.

\smallskip

\noindent Define the first homotopy $j$ by $j_{0}=\mathrm{identity}$
and for $s\in(0,1]$ and all $\omega\in \mathrm{pop}(X_{b},X,X_{a})$
$$j_{s}(\omega)(t)=\left\{\begin{array}{ll}r\left(\omega(t\cdot\frac{\beta^{-1}_{\omega}(s)}{\beta^{-1}_{\omega}(\frac{1}{2}s)}),1\right)&0\leq t\leq \beta^{-1}_{\omega}(\frac{1}{2}s)
\\R\left(\omega(\beta^{-1}_{\omega}(s))\right)(\frac{t-\beta^{-1}_{\omega}(\frac{1}{2}s)}{\beta^{-1}_{\omega}(s)-\beta^{-1}_{\omega}(\frac{1}{2}s)})&\beta^{-1}_{\omega}(\frac{1}{2}s)\leq t\leq \beta^{-1}_{\omega}(s)
\\\omega(t)&\beta^{-1}_{\omega}(s)\leq t\leq T_{\omega}
\end{array}\right..$$

\noindent Intuitively this homotopy retracts the start of the path
back into $X_{a}$ and then travels along the reverse of the tameness
retraction of $\omega(\beta^{-1}_{\omega}(s))$ before continuing
along the original path from then onwards.

\smallskip

\noindent For the second homotopy $k$ again define
$k_{0}=\mathrm{identity}$ and for $s\in(0,1]$ and all $\omega\in
\mathrm{pop}(X_{b},X,X_{a})$
$$k_{s}(\omega)(t)=\left\{\begin{array}{ll}\omega(t)&0\leq t\leq \beta^{-1}_{\omega}(\frac{1}{2}s)
\\g_{s}\left(\omega|_{[\beta^{-1}_{\omega}(\frac{1}{2}s),T_{\omega}]}\right)(t)&\beta^{-1}_{\omega}(\frac{1}{2}s)\leq t\leq T_{\omega}
\end{array}\right.,$$
where $g$ is the homotopy constructed in the previous lemma with
respect to $\{\sigma\in
\mathrm{pop}(X_{b},X-X_{a},X-X_{a})|\sigma(\delta)\not\in X_{c}$ for
all $\delta>0\}\subset \mathrm{pop}(X_{b},X-X_{a},X-X_{a})$ which
can used because of the inductive hypothesis. This homotopy ensures
that for $s>0$, $k_{s}(\omega)(t)\not\in X_{c}$ for
$t\in(\beta^{-1}_{\omega}(\frac{1}{2}s),T_{\omega}]$.

\smallskip

\noindent Now if we define a homotopy
$l:\mathrm{pop}(X_{b},X,X_{a})\times I\rightarrow
\mathrm{pop}(X_{b},X,X_{a})$ by $l_{s}=j_{s}\circ k_{s}$ we get
$\mathrm{image} \{l_{s}\}\subset \mathrm{pop}(X_{b},X-X_{c},X_{a})$
for $s>0$ because the start of the path is the retraction to $X_{a}$
by a nearly stratum preserving retraction of the point $g_{s}\left
(\omega|_{[\beta^{-1}_{\omega}(\frac{1}{2}s),T_{\omega}]}\right
)(\beta^{-1}_{\omega}(s))$ which is not in $X_{c}$ or a lower
strata. Clearly $l_{0}=\mathrm{identity}$ because
$j_{0}=k_{0}=\mathrm{identity}$ and $l$ fixes start and end points
of paths because both $j$ and $k$ fix start and end points of paths.

\smallskip

\noindent Thus we have satisfied the requirements for
$\mathrm{pop}(X_{b},X-X_{c},X_{a})\subset
\mathrm{pop}(X_{b},X,X_{a})$ to be a special path inclusion and so
using induction to remove strata like $X_{c}$ we reach the two
strata case and so conclude $\mathrm{hol}(X_{b}\cup
X_{a},X_{a})\subset \mathrm{pop}(X_{b},X,X_{a})$ is a special path
inclusion by Lemma 2.5.\qedhere

\end{proof}

\

\

\end{document}